\documentclass{amsart}
\usepackage{bbm}
\usepackage{verbatim}
\usepackage[dvips]{graphics}

\newtheorem{theorem}{Theorem}[section]
\newtheorem{proposition}{Proposition}[section]
\newtheorem{lemma}[theorem]{Lemma}
\newtheorem{corollary}[theorem]{Corollary}

\theoremstyle{remark}

\newcommand\ind[1]{\mathbbm{1}_{\{#1\}}}

\def\E{\mathbb{E}}

\def\N{\mathbb{N}}
\def\P{\mathbb{P}}
\def\R{\mathbb{R}}

\def\cal{\mathcal}
\def\etal{{ et al.}}

\def\eps{\varepsilon}

\title[Interacting branching processes]{Interacting branching processes and linear file-sharing networks}

\author[L. Leskel\"a]{Lasse Leskel\"a}
\address[L. Leskel\"a]{
Aalto University,
Department of Mathematics and Systems Analysis,
PO Box 11100, 00076 Aalto, Finland
}
\email{lasse.leskela@iki.fi}
\urladdr{http://www.iki.fi/lsl}

\author[Ph. Robert]{Philippe Robert}\thanks{Work partially supported by SCALP Project funded by
EEC Network of Excellence  Euro-FGI, and the Academy of Finland}
\author[F. Simatos]{Florian Simatos}
\address[Ph. Robert, F. Simatos]{INRIA Paris --- Rocquencourt, Domaine de Voluceau, BP 105,  
78153   Le Chesnay, France. } 
\email{Philippe.Robert@inria.fr}
\email{Florian.Simatos@inria.fr}
\urladdr{http://www-rocq.inria.fr/\string~robert}


\date{\today}

\begin{document}

\begin{abstract}
File-sharing networks are  distributed systems used to disseminate files  among nodes of a
communication network. The  general simple principle of these systems is  that once a node
has retrieved a file, it may become a server for this file.  In this paper, the capacity
of these  networks is analyzed with  a stochastic model when  there is a  constant flow of
incoming  requests for a  given file.  It is  shown that  the  problem
can be solved by analyzing the  asymptotic  behavior of  a  class  of interacting  branching
processes.  Several  results of independent interest concerning  these 
branching processes are derived and then used to study the file-sharing systems.
\end{abstract}

\maketitle

\noindent\rule{\textwidth}{0.2mm}
\vspace{-1cm}
\tableofcontents
\vspace{-1.3cm}
\noindent\rule{\textwidth}{0.2mm}

\section{Introduction}
File-sharing  networks are  distributed systems  used to  disseminate information  among 
some nodes of a communication network. The general simple principle is the following: once a
node has retrieved  a file it becomes  a server for this  file. An improved
version of  this principle  consists in  splitting the original  file into  several pieces
(called ``chunks'') so that a given node can retrieve simultaneously several chunks of the
same file  from different servers.  In this case,  the rate to get  a given file  may thus
increase significantly as well as the global capacity of the  file-sharing system: a
node becomes a server of  a chunk as soon as it has retrieved it  and not only when it has
the whole file.  These schemes disseminate information efficiently as long as the number of servers is growing rapidly.

The paper  investigates the maximal throughput of these file-sharing  networks, i.e., if
the system can  cope or not with a constant flow of  incoming requests  for a given
file. Two cases are considered, either the file consists of one chunk or the file is split
in two  chunks which  are retrieved sequentially, which we will refer to as a linear
file-sharing network.   It is  assumed that arrival  times are
Poisson and chunk transmission times  are exponentially distributed.  In this setting, the
stability property of  the file-sharing system is expressed as  the ergodicity property of
the associated  Markov process.  Even in  this simple framework,  mathematical studies are
quite  scarce,  see  Qiu  and  Srikant~\cite{Qiu04:0},  Simatos  \etal~\cite{Simatos08:0},
Susitaival \etal~\cite{Susitaival06:0} and references therein.

The main  technical difficulty  in proving stability/instability  results for this  class of
stochastic networks is  that, except for the input, the Markov  process has unbounded jump
rates,  in fact  proportional to  one of  the coordinates  of the  current state.   If $x$
servers have  a chunk, since each  of them can  deliver this chunk, they  globally provide
this chunk  at a  rate proportional  to $x$. For this reason, the classical tools related
to fluid limits cannot be used easily in this setting. See Bramson~\cite{Bramson},  Chen  and Yao~\cite{Chen:14}  and
Robert~\cite{Robert03} for example. However, this class of processes is close  to another
important class  of stochastic processes, namely branching  processes, where a  population of size
$x$ evolves at rate proportional to $x$. As it will be seen, a file sharing system
distributing a file split in  several  chunks can  be  represented  as a  Markov  process
associated to  multi-type branching processes with interaction.

\noindent
{\bf Interacting Branching  Processes.}  Consider the case of a file  split into two chunks,
one considers the  case when  a new user arriving  in the network  requests first
chunk number~$1$  and then  chunk number~$2$. We further assume that a  server of type-$1$,  i.e., having  chunk $1$
only, distributes it at  rate $\mu_1$, while a server of type-$2$, having  chunks $1$ and $2$,
distributes  only chunk~$2$ at  rate $\mu_2$. For $i=0$, $1$, $2$, let $X_i(t)$ be the 
number of  type-$i$ servers at time $t$. A type-$0$ server is simply a user without any
chunk. The crucial observation is that as long as there
are requests for chunk~$i\in\{1,2\}$, then $(X_i(t))$ grows similarly as a branching process where
individuals give birth to one child at rate $\mu_i$, since each server offers a capacity
$\mu_i$ for the requested chunk. On the other hand, a new arrival in $X_2$ corresponds to a
departure in $X_1$, since the new type-$2$ server was a type-$1$ server, thus deaths in $X_1$ are governed by births in $X_2$. 
\begin{figure}[ht]
\begin{center}
\scalebox{.5}{\includegraphics{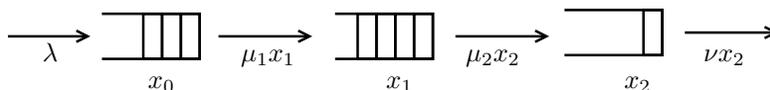}}
\put(-280,0){$ \lambda $}
\put(-240,-10){$ x_0 $}
\put(-205,0){$ \mu_1x_1 $}
\put(-150,-10){$ x_1 $}
\put(-120,0){$ \mu_2x_2 $}
\put(-60,-10){$ x_2 $}
\put(-30,0){$ \nu x_2 $}
\end{center}
\caption{Transition rates outside the boundaries of the file-sharing network with two chunks.}\label{chunk-fig}
\end{figure}

The file-sharing network under consideration can thus be seen as a system of
interacting  branching  processes where  the  births and  the  deaths  of individuals  are
correlated.  In the  simpler  case of  one chunk,  a  description as  a branching  process
(without  interaction)  has  been  used   to  analyze  the  transient  behavior  of  these
file-sharing systems.  See Yang and de Veciana~\cite{Yang04:0}, Dang \etal~\cite{Dang} and
Simatos \etal~\cite{Simatos08:0}.

Sections~\ref{sec:Yule} and~\ref{sec:Interacting} present  results of independent interest
concerning branching processes  where individuals are killed at  the instants $(\sigma_n)$
of a random point process. In Section~\ref{sec:Yule} a criterion for the extinction of the
branching process  is obtained  in terms  of the sequence  $(\sigma_n)$ and  an asymptotic
result is  derived in this  case.  Section~\ref{sec:Interacting} considers the  case where
$(\sigma_n)$  is  the  sequence  of   birth  instants  of  another  independent  branching
process. Several useful  estimates are derived in this setting. These  results are used to
establish  the stability  results concerning  file-sharing systems  with two  chunks.  The
stability  properties of a  network with  a single-chunk  file are  analyzed in  detail in
Section~\ref{sec:two-queue}. The case of file-sharing networks with two chunks is detailed
in Section~\ref{sec:net}.

\medskip

\noindent
{\bf Acknowledgements.} \\
This paper has benefited from various interesting discussions with  S. Borst, I. Norros,
R. N\'u\~{n}ez-Queija, B.J. Prabhu, and  H. Reittu. 

\section{Yule Processes with Killing} \label{sec:Yule}

A Yule process $(Y(t))$ with rate $\mu>0$ is a Markovian branching process with $Q$-matrix 
\begin{equation}\label{Yuledef}
q_Y(x,x+1)=\mu x, \quad\forall x\geq 0.
\end{equation}
A Yule process is simply a pure birth process, where each individual gives birth to a child at rate $\mu$.

Yule processes are the basic branching processes that appear in analyzing the two-chunk
network of Section~\ref{sec:net}. Actually a variant of this stochastic model will be
needed, where some individuals are killed. In this section we study this model when
killings  are given by an exogenous process and occur 
at fixed (random or deterministic) epochs; in Section~\ref{sec:Interacting} killings
result from the interaction with another branching process. 

In terms of branching processes,  this killing procedure amounts to prune
the tree,  i.e., to  cut some  edges of  the tree, and  the subtree  attached to  it. This
procedure is  fairly common  for branching processes,  in the Crump-Mode-Jagers  model for
example,   see  Kingman~\cite{Kingman}.   See  also   Neveu~\cite{Neveu}  or   Aldous  and
Pitman~\cite{Aldous}.

\medskip
\subsection*{A Yule Process Killed at Fixed Instants}
Until the end of this section, it is assumed  that, provided that  it is non-empty, at  epochs $\sigma_n$,
$n\geq  1$, an  individual is  removed from  the population  of an  ordinary  Yule process
$(Y(t))$  with  rate  $\mu_W$ starting  with  $Y(0) = w \in \N$  individuals.  It is  assumed  that
$(\sigma_n)$ is  some fixed non-decreasing  sequence.  It will  be shown that  the process
$(W(t))$ obtained by killing one individual  of $(Y(t))$ at each of the successive instants
$(\sigma_n)$  survives  with  positive  probability  when the  series  with  general  term
$(\exp(-\mu_W \sigma_n))$ converges.

We denote
\[
\kappa = \inf\{n \geq 1: W(\sigma_n)=0\}.
\]
The process $(W(t))$ can be represented by
\begin{equation} \label{eq:W}
W(t)=Y(t)-\sum_{i=1}^{\kappa} X_i(t)\ind{\sigma_i\leq t},
\end{equation}
where, for $1\leq i\leq \kappa$ and $t\geq \sigma_i$,  $X_i(t)$ is the total number of
children at time $t$ in the original Yule process of the $i$th individual killed at time
$\sigma_i$.  
In terms of trees, $(W(t))$ can be seen as a subtree of $(Y(t))$: for $1 \leq i \leq
\kappa$, $(X_i(t))$ is the subtree of $(Y(t))$ associated with  the $i$th particle killed  at time
$\sigma_i$.  

It is easily checked that 
$(X_i(t-\sigma_i), t\geq \sigma_i)$ is a Yule 
process  starting with one individual and,  since a killed individual cannot have one of his
descendants killed, that the processes 
\[
(\widetilde{X}_i(t))=(X_i(t+\sigma_i), t\geq 0), \quad 1 \leq i \leq \kappa,
\]  
are  independent Yule processes. 

For any process $(U(t))$, one denotes
\begin{equation}
(M_{U}(t))\stackrel{\text{def.}}{=} \left(e^{-\mu_W t} U(t)\right).
\end{equation}
If $(\widetilde{X}(t))$ is a Yule process with rate $\mu_W$,  the martingale $(M_{\widetilde{X}}(t))$
converges almost surely and in $L_2$ to a random variable $M_{\widetilde{X}}(\infty)$ with an
exponential distribution with mean $\widetilde{X}(0)$, and by Doob's Inequality
\[
\E\left(\sup_{t\geq 0} M_{\widetilde{X}}(t)^2 \right)\leq 2\sup_{t\geq 0} \E\left(M_{\widetilde{X}}(t)^2\right)<+\infty.
\]
See Athreya and Ney~\cite{Athreya72:0}. Consequently
\[ 
e^{-\mu_W t}W(t)= M_Y(t)-\sum_{i=1}^{\kappa} e^{-\mu_W \sigma_i} M_{\widetilde{X}_i}(t-\sigma_i)\ind{\sigma_i\leq t},
\]
and for any $t\geq 0$, 
\[
\sum_{i=1}^{\kappa} e^{-\mu_W \sigma_i} M_{\widetilde{X}_i}(t-\sigma_i)\ind{\sigma_i\leq  t}\leq 
\sum_{i=1}^{\kappa} e^{-\mu_W \sigma_i} \sup_{s\geq 0}M_{\widetilde{X}_i}(s).
\]
Assume now that $\sum_{i \geq 1} e^{-\mu_W \sigma_i} < +\infty$: then the last expression
is integrable, and Lebesgue's Theorem implies that  $(M_{W}(t)) = (\exp(-\mu_W t)W(t))$ converges almost
surely and in $L_2$ to 
\[
M_W(\infty)=M_Y(\infty)-\sum_{i=1}^{\kappa} e^{-\mu_W \sigma_i}M_{\widetilde{X}_i}(\infty).
\]
Clearly, for some $w^*$ large enough and then for any $w \geq w^*$, one has 
\[
\E_w(M_W(\infty))\geq w - \sum_{i=1}^{+\infty} e^{-\mu_W \sigma_i} > 0,
\]
in particular $\P_w(M_W(\infty)>0)>0$ and $\P_w(W(t)\geq 1, \forall t\geq 0)>0$. If $Y(0) = w < w^*$ and $\sigma_1>0$, then $\P_w(Y(\sigma_1)\geq w^* + 1) > 0$ and therefore, by translation
at time $\sigma_1$, the same conclusion holds when the sequence $(\exp(-\mu_W
\sigma_i))$ has a finite sum. The following proposition has thus been proved. 
\begin{proposition} \label{KillProp}
Let $(W(t))$ be a process growing as a Yule process with rate $\mu_W$ and for
which individuals are killed at non-decreasing instants $(\sigma_n)$ with $\sigma_1>0$. If
\[
\sum_{i=1}^{+\infty} e^{-\mu_W \sigma_i}<+\infty,
\]
then as $t$ gets large, and for any $w \geq 1$, the variable $(\exp(-\mu_W t)W(t))$ converges $\P_w$-almost surely and in $L_2$
to a finite random variable $M_W(\infty)$ such that $\P_w(M_W(\infty)>0)>0$.
\end{proposition}

The previous proposition establishes  the minimal results needed in Section~\ref{sec:net}.
However, Kolmogorov's  Three-Series, see  Williams~\cite{Williams91:0}, can be
used in conjunction with  Fatou's Lemma to show that $(W(t))$ dies  out almost surely when
the series with general term $(\exp(-\mu_W \sigma_n))$ diverges.

\section{Interacting Branching Processes} \label{sec:Interacting}

In this section we study the Yule process of the previous section when killing times
correspond to the birth times of some other branching process. The other branching process
can be seen as a renewing Bellman-Harris process. 

\subsection*{Renewing Bellman-Harris process}

In the rest of this section, $\mu_Z, \nu > 0$ are fixed and $(Z(t))$ is a birth-and-death
process whose $Q$-matrix $Q_Z$ is given by \begin{equation}\label{eqZ} q_Z(z,z+1)=\mu_Z
  (z\vee 1)\ \text{ and }\ q_Z(z,z-1)=\nu z.  
\end{equation}
In the rest of the paper $n \vee m$ denotes $\max(n,m)$ for $n$, $m \in \N$. This process
can be described equivalently as a time-changed $M/M/1$ queue (see
Proposition~\ref{prop:queueing}) or as a sequence of independent branching processes (see
Proposition~\ref{prop:branching}). The time-change is the discrete analog of the Lamperti
transform between continuous-state branching processes and L\'evy processes, see
Lamperti~\cite{Lamperti67:1}. As it will be seen these two viewpoints are complementary. 

Let $(\sigma_n)$ be the sequence of birth instants (i.e., positive jumps) of $(Z(t))$ and
$(B_\sigma(t))$ the corresponding counting process of $(\sigma_n)$, for $t \geq 0$,  
\[
B_\sigma(t) =\sum_{i\geq 1}\ind{\sigma_i \leq t}.
\]
\begin{proposition}[Queueing Representation] \label{prop:queueing}
If $Z(0) = z \in \N$, then 
\begin{equation}\label{aux}
(Z(t), t\geq 0 )\stackrel{\text{dist.}}{=} \left(L(C(t)), t\geq 0\right),
\end{equation}
where $(L(t))$ is the process of the number of jobs of an $M/M/1$ queue with input rate
$\mu_Z$ and service rate $\nu$ and with $L(0) = z$ and
$C(t)=\inf\left\{s>0:A(s)>t \right\}$, where
\[
A(t)=\int_0^t\frac{1}{1\vee L(u)}\,du .
\]
\end{proposition}
\begin{proof}
It is not difficult to check that the process
$(M(t))\stackrel{\text{def.}}{=}\left(L(C(t))\right)$ has the Markov property. Let $Q_M$
be its $Q$-matrix. For $z\geq 0$,
\[
\P(L(C(h))=z+1\mid L(0)=z)=\mu_Z \E(C(h))+o(h)=\mu_Z (z\vee 1)h+ o(h),
\]
hence $q_M(z,z+1)=\mu_Z(z\vee 1)$. Similarly $q_M(z,z-1)=\nu z$. The proposition is proved.  
\end{proof}

\begin{corollary}\label{cor:series}
	For any $\gamma > (\mu_Z - \nu) \vee 0$ and  $z = Z(0) \in \N$,
	\begin{equation} \label{eq:series}
	\E_z\left(\sum_{n=1}^{+\infty} e^{-\gamma \sigma_n}\right) < +\infty.
	\end{equation}
\end{corollary}

\begin{proof}
Proposition~\ref{prop:queueing} shows that, in particular, the sequences of positive jumps
of  $(Z(t))$   and  of   $(L(C(t)))$  have  the   same  distribution.  Hence,   if  ${\cal
  N}_{\mu_Z}=(t_n)$ is  the arrival  process of  the $M/M/1$ queue,  a Poisson  process with
parameter   $\mu_Z$,   then,   with   the   notations   of   the   above   proposition,   the
relation   $$(\sigma_n)\stackrel{\text{dist.}}{=}(A(t_n))$$  holds.   By   using  standard
martingale  properties of  stochastic integrals  with  respect to  Poisson processes,  see
Rogers and Williams~\cite{Rogers2}, one gets for $t\geq 0$,
\begin{align}
\E_z\left(\sum_{n  \geq 1}  e^{-\gamma  A(t_n)}\right)&=\E_z\left(\int_0^\infty e^{-\gamma
  A(s)}   {\cal   N}_{\mu_Z}(ds)\right)=   \mu_Z\E_z\left(\int_0^\infty  e^{-\gamma   A(s)}   \,
ds\right)\notag  \\  &=   \mu_Z\int_0^\infty  e^{-\gamma  u}\E_z\left(Z(u)\vee  1  \right)\,
du,\label{auxj}
\end{align}
where Relation~\eqref{aux} has been used for the last equality. 
Kolmogorov's equation for the process  $(Z(t))$ gives that
\begin{align*}
\phi(t)\stackrel{\text{def.}}{=}\E_z(Z(t))&=\mu_Z\int_0^t     \E_z\left(Z(u)\vee    1
\right)\,du -\nu\int_0^t \E_z\left(Z(u)\right)\,du\\
&\leq (\mu_Z-\nu)\int_0^t \phi(u)\,du +\mu_Z t,
\end{align*}
therefore, by Gronwall's Lemma, 
	\[
	\phi(t)\leq \phi(0)+ \mu_Z \int_0^t ue^{(\mu_Z-\nu)u}\,du \leq z + \frac{\mu_Z}{\mu_Z-\nu} t e^{(\mu_Z-\nu)t}.
	\]
 From Equation~\eqref{auxj}, one concludes that
	\[
	\E_z\left(\sum_{n} e^{-\gamma \sigma_n}\right)=
	\E_z\left(\sum_{n} e^{-\gamma A(t_n)}\right)<+\infty. 
	\]
	The proposition is proved.
\end{proof}

Before  hitting   $0$,     $(Z(t))$   can  be  seen  a  Bellman-Harris  branching  process with Malthusian   parameter  $\alpha   =   \mu_Z  -\nu$, see  Athreya   and
Ney~\cite{Athreya72:0}. This Bellman-Harris branching process describes  the evolution of  a population of
independent particles: each particle, at rate $\lambda \stackrel{\text{def.}}{=} \mu_Z + \nu$, either splits into two  particles with probability  $p \stackrel{\text{def.}}{=}
\mu_Z  /  (\mu_Z  +   \nu)$  or  dies with probability $1-p$.  These  processes  will  be   referred  to  as    $(p,
\lambda)$-branching processes in the sequel.

A $(p, \lambda)$-branching process survives with positive probability only when $p > 1/2$,
in which case the probability of extinction $q$ is equal to $q = (1-p)/p = \nu / \mu_Z$. The process $(Z(t))$ only differs from a $(p,\lambda)$-branching process insofar it regenerates after hitting
$0$, after a time exponentially distributed. When it regenerates, it again behaves as a $(p,\lambda)$-branching process (started
with one particle), until it hits~$0$ again.  

\begin{proposition}[Branching Representation] \label{prop:branching}
If $Z(0) = z \in \N$ and  $(\widetilde Z(t))$ is a $(p,\lambda)$-branching
process started with $z \in \N$ particles and $\widetilde T$ its extinction time, then 
\[
(Z(t), 0 \leq t \leq T) \stackrel{\text{dist.}}{=} (\widetilde Z(t), 0 \leq t \leq \widetilde T),
\]
where $T = \inf\{t \geq 0: Z(t) = 0\}$ is the hitting time of $0$ by $(Z(t))$.
\end{proposition}

\begin{corollary}\label{cor:Z(infty)}
	Suppose that $\mu_Z > \nu$. Then $\P_z$-almost surely for any $z \geq 0$, there
        exists a finite random variable $Z(\infty)$ such that, 
	\[
	\lim_{t \to +\infty}  e^{-(\mu_Z-\nu)t}Z(t)  = Z(\infty)\ \text{ and }\ Z(\infty) > 0.
	\]         
\end{corollary}

\begin{proof}
	When $\mu_Z > \nu$, the process $(Z(t))$ couples in finite time with a supercritical $(p,
        \lambda)$-branching process $(\widetilde Z(t))$ conditioned on non-extinction; this follows readily from
        Proposition~\ref{prop:branching} (or see the Appendix for details). Since for any
        supercritical $(p, \lambda)$-branching process, $(\exp(-(\mu_Z-\nu) t)\widetilde Z(t))$ converges almost surely to a finite
        random variable $\widetilde Z(\infty)$, positive on the event of non-extinction
        (see Nerman~\cite{Nerman81:0}), one gets the desired result.
\end{proof}

Due to its technicality, the proof of the following result is postponed to the Appendix;
this result is used in the proof of Proposition~\ref{prop:Z(H0)}. 

\begin{proposition}\label{lemma:births}
Suppose that $\mu_Z > \nu$. If
\begin{equation} \label{eq:eta}
	\eta^*(x) = \frac{2 - x - \sqrt{x(4-3x)}}{2(1-x)},\ 0 < x < 1,
\end{equation}
then for any $0 < \eta < \eta^*(\nu/\mu_Z)$,
\[ \sup_{z \geq 0} \left[ \E_z \left( \sup_{t \geq \sigma_1} \left( e^{\eta (\mu_Z - \nu) t} B_\sigma(t)^{-\eta} \right) \right) \right] < +\infty. \]
\end{proposition}

\subsection*{A Yule Process Killed at the Birth Instants of a Renewing Bellman-Harris  Process}

We now consider the main interacting branching processes: in addition to $(Z(t))$, one  considers an independent Yule process $(Y(t))$  with parameter  $\mu_W$ (its $Q$-matrix is defined  by Relation~\eqref{Yuledef} with $\mu = \mu_W$). We proceed similarly as in Section~\ref{sec:Yule}: a process $(W(t))$ is defined 
by killing one individual of $(Y(t))$ at each of  the  birth instants $(\sigma_n)$  of $(Z(t))$ (see Alsmeyer~\cite{Alsmeyer93:0} and the references therein for related models). Recall that $(W(t))$ is given by Formula~\eqref{eq:W}. The following results are key to analyzing the two-chunk network of Section~\ref{sec:net}.

\begin{proposition}\label{prop:Z(H0)}
Assume that $\mu_Z - \nu > \mu_W$, and let $H_0$  be the extinction time of
$(W(t))$, i.e., 
\[ 
H_0 = \inf \{t \geq 0: W(t) = 0 \}, 
\]             
then the random variable $H_0$ is almost surely finite  and:
\begin{enumerate}\renewcommand{\labelenumi}{(\roman{enumi})}
\item \label{it1} $Z(H_0) - Z(0) \leq e^{\mu_W H_0} M_{Y}^*$ where 
\[
M_{Y}^* = \sup_{t \geq 0} \left( e^{-\mu_W t}  Y(t) \right).
\]
\item There exists a finite constant $C$ such that for any $z \geq 0$ and $w \geq 1$,
	\begin{equation} \label{eq:H_0}
		\E_{(w,z)}(H_0) \leq C \left( \log(w) + 1\right).
	\end{equation}
\end{enumerate}
\end{proposition}

In~\eqref{eq:H_0} the subscript $(w,z)$ refers to the initial state of the Markov process
$(W(t), Z(t))$. More generally in the rest of the paper we use the convention that if
$(U(t))$ is a Markov process then the index $u$ of $\P_u$ and $\E_u$ will refer to the
initial condition of this Markov process. 

\begin{proof}
Define $\alpha=\mu_Z - \nu$.
Concerning the almost sure finiteness of $H_0$, note that Equation~\eqref{eq:W} entails
that $W(t) \leq Y(t) - B_\sigma(t)$ for all $t \geq 0$ on the event $\{H_0 =
+\infty\}$. As $t$ goes to infinity, both $\exp(-\mu_W t) Y(t)$ and $\exp(- \alpha t)
B_\sigma(t)$ converge almost surely to positive and finite random variables (see
Nerman~\cite{Nerman81:0}), which implies, when $\alpha=\mu_Z-\nu > \mu_W$, that $W(t)$ converges
to~$-\infty$ on $\{H_0 = +\infty\}$, and so this event is necessarily of probability zero. 
	
The first point~(i) of the proposition comes from Identity~\eqref{eq:W} evaluated at $t = H_0$: 
\begin{equation}\label{NT}
	Z(H_0) -Z(0) \leq  B_\sigma(H_0) \leq  Y(H_0) \leq  e^{\mu_W H_0} M_{Y}^*.
\end{equation}
 By using the relation $\exp(x) \geq x$, Equation~\eqref{eq:H_0} follows from the following bound:       
for any $\eta < \eta^*(\nu / \mu_Z)$ (recall that $\eta^*$ is given   by Equation~\eqref{eq:eta}), 
	\begin{equation} \label{eq:exp(H_0)}
		\sup_{w \geq 1, z \geq 0} \left[ w^{-\eta} \E_{(w,z)}\left( e^{\eta(\alpha - \mu_W)H_0} \right) \right] < +\infty.
	\end{equation}
So all  is left to  prove is this  bound. Under $\P_{(w,z)}$, $(Y(t))$  can be represented
as the sum  of $w$ i.i.d.\  Yule  processes, and  so  $M_{Y}^*  \leq M^*_{Y,1}  +  \cdots
+ M^*_{Y,w}$  with $(M^*_{Y,i})$  i.i.d.\ distributed  like  $M_{Y}^*$ under  $\P_{(1,z)}$;
Inequality~\eqref{NT}  then entails that
	\[
	e^{(\alpha - \mu_W) H_0} \leq \left(\sum_{i=1}^w M^*_{Y,i}\right) \times \sup_{t \geq \sigma_1} \left( e^{\alpha t}  / B_\sigma(t) \right).
	\]
By independence of $(M^*_{Y,i})$ and~$(B_\sigma(t))$, Jensen's inequality gives for any $\eta < 1$
	\[
	\E_{(w,z)} \left( e^{\eta(\alpha - \mu_W) H_0} \right) \leq w^\eta \left( \E\left( M^*_{Y,1} \right)\right)^\eta \E_z \left( \sup_{t \geq \sigma_1} \left( e^{\eta \alpha t} B_\sigma(t)^{-\eta} \right) \right),
	\]
hence the bound~\eqref{eq:exp(H_0)} follows from Proposition~\ref{lemma:births}.
\end{proof}

\noindent
One concludes this section with a Markov chain which will be used in Section~\ref{sec:net}.
Define recursively the sequence $(V_n)$ by, $V_0=v$ and
\begin{equation}\label{eqV}
V_{n+1} {=} \sum_{k=1}^{A_n(V_n)} I_{n, k}, \;\; n\geq 0,
\end{equation}
where for each $n$  $(I_{n,k},k \geq 1)$ are  identically  distributed  integer valued random variables  independent  of $V_n$  and
$A_n(V_n)$,  and such  that $\E(I_{n, 1})=p$  for some  $p\in(0,1)$. For  $v>0$, $A_n(v)$  is an
independent random variable with the  same distribution as $Z(H_{0})$ under $\P_{(1,v)}$, i.e., with  the initial condition $(W(0),
Z(0))=(1,v)$.

The above equation~\eqref{eqV} can be interpreted as a branching  
process with immigration, see Seneta~\cite{Seneta70:0}, or also as an auto-regressive model. 
\begin{proposition}\label{prop:V}
Under the condition $\mu_Z - \nu > \mu_W$, if $(V_n)$ is the Markov chain defined by Equation~\eqref{eqV}
and, for $K \geq 0$,  
\[
	N_K = \inf\{n \geq 0: V_n \leq K\},
	\]
then there exist $\gamma > 0$ and $K\in\N$  such that 
	\begin{equation} \label{eq:N}
		\E(N_K | V_0 = v) \leq \frac{1}{\gamma}\log(1+v), \quad\forall v\geq 0.
	\end{equation}
The Markov chain $(V_n)$ is in particular positive recurrent.
\end{proposition}
\begin{proof}
For $V_0=v\in\N$, Jensen's Inequality and Definition~\eqref{eqV} give the relation
\begin{equation}\label{Goisot}
\E_v\log \left(\frac{1+V_1}{1+v}\right)\leq \E_{(1,v)}\log \left[\frac{1+pZ(H_{0})}{1+v}\right].
\end{equation}
From Proposition~\ref{prop:Z(H0)} and by using the same notations, one gets that, under
$\P_{(1,v)}$, 
\[
Z(H_{0})\leq v+ e^{\mu_W H_0}M_{Y}^*,
\]
where $(Y(t))$  is a Yule process starting with one individual. 
By looking at the  birth instants of $(Z(t))$, it is easily checked that
 the random variable $H_{0}$ under $\P_{(1,v)}$  is stochastically
bounded by $H_{0}$ under $\P_{(1,0)}$. The integrability of $H_0$ under $\P_{(1,0)}$ (proved in
Proposition~\ref{prop:Z(H0)}) and of $M_{Y}^*$ give that the expression
\[
\log\left(\frac{1+p(v+ e^{\mu_W H_0}M_{Y}^*)}{1+v}\right)
\]
bounding the right hand side of Relation~\eqref{Goisot} is also an integrable random variable under $\P_{(1,0)}$.
Lebesgue's Theorem gives  therefore that 
\[
\limsup_{v\to+\infty} \left[ \E_v\log \left(\frac{1+V_1}{1+v}\right) \right] \leq \log p <0.
\]
Consequently, one concludes that $v \mapsto \log(1+v)$ is a Lyapunov function for the Markov
chain $(V_n)$, i.e., if $\gamma=-(\log p)/2$, there exists $K$ such that for $v\geq K$,
\[
\E_v\log \left(1+V_1\right) - \log \left(1+v\right)\leq -\gamma. 
\]
Foster's criterion, see Theorem~8.6 of Robert~\cite{Robert03}, implies
that $(V_n)$ is indeed ergodic and that Relation~\eqref{eq:N} holds. 
\end{proof}
\section{The  Single-Chunk Network} \label{sec:two-queue}
This  section  is devoted  to  the  study of  a  class  of  two-dimen\-sional Markov  jump
processes $(X_0(t),X_1(t))$,  the corresponding $Q$-matrix $\Omega_r$ is given, for $x=(x_0,x_1)\in\N^2$, by
\begin{equation}\label{def}
\begin{cases}
\Omega_r[(x_0,x_1),(x_0+1,x_1)]&=\ \lambda,\\
\Omega_r[(x_0,x_1),(x_0-1,x_1+1)]&=\ \mu r(x_0,x_1) (x_1\vee 1)\ind{x_0 > 0},\\
\Omega_r[(x_0,x_1),(x_0,x_1-1)]&=\ \nu x_1,
\end{cases}
\end{equation}
where $x \mapsto r(x)$, referred to as the {\em rate function},
is some fixed function on $\N^2$ with values in $[0,1]$.

\medskip
From a modeling perspective, this Markov process describes the following system.  Requests
for a single file  arrive with rate $\lambda$, the first component  $X_0(t)$ is the number
of requests  which did not  get the file,  whereas the second  component is the  number of
requests having the file and acting  as servers until they leave the file-sharing network.
The constant $\mu$ can  be viewed as the file transmission rate, and  $\nu$ as the rate at
which servers having the file leave.  The term $r(x_0,x_1)$ describes the interaction of
downloaders and uploaders in  the system. The term $x_1\vee 1$ can  be interpreted so that
there is one  permanent server in the  network, which is contacted if  there are no
other uploader nodes in  the system. A related system where there  is always one permanent
server for the file can be modeled by replacing the term $x_1\vee 1$ by $x_1 + 1$. See the
remark at the end of this section.

Several related examples of this class of models have been recently investigated. The case
\[
r(x_0,x_1) = \frac{x_0}{x_0 + x_1}
\]
is considered in N\'u\~nez-Queija and Prabhu~\cite{Queija08:0} and Massouli\'e and
Vojnovi\'c~\cite{Massoulie05:0}; in this case the downloading time of the file is
neglected. 
Susitaival \etal~\cite{Susitaival06:0} analyzes the rate  function $r(x)$ 
\[
r(x_0,x_1)= 1 \wedge \left(\alpha \frac{x_0}{x_1}\right)
\] 
with $\alpha > 0$ and $a\wedge b$ denotes $\min(a,b)$ for $a$, $b\in\R$.
This model allows to take into account that a request cannot be served by more
than one server.  See also Qiu and Srikant~\cite{Qiu04:0}. 

With a slight abuse of notation, for $0<\delta\leq 1$, the matrix $\Omega_\delta$ will refer to
the case when the function $r$ is identically  equal  to $\delta$.   Note  that  the
boundary  condition  $x_1  \vee 1$  for departures  from the first  queue prevents  the
second  coordinate from  ending up  in the absorbing state~$0$. Other possibilities are
discussed at the end of this section. In the following $(X^r(t))=(X^r_0(t),X^r_1(t))$
[resp.\ $(X^\delta(t))$] will denote a Markov process with $Q$-matrix $\Omega_r$
[resp.\ $\Omega_\delta$].   

\subsection*{Free Process}
For $\delta>0$, $Q_\delta$ denotes the following $Q$-matrix 
\begin{equation}\label{defQd}
\begin{cases}
Q_\delta[(y,z),(y+1,z)]&=\lambda,\\
Q_\delta[(y,z),(y-1,z+1)]&=\mu \delta  (z\vee 1),\\
Q_\delta[(y,z),(y,z-1)]&=\nu z.
\end{cases}
\end{equation}
The  process  $(Y^\delta(t))=(Y_0^\delta(t),  Y_1^\delta(t))$,  referred to  as  the  free
process, will  denote a Markov  process with $Q$-matrix  $Q_\delta$.  Note that  the first
coordinate $(Y_0^\delta(t))$  may become negative.  The second  coordinate $(Y_1^\delta(t))$ of
the free  process is the birth-and-death process with the same distribution as the process
$(Z(t))$ introduced in Section~\ref{sec:Interacting}.  
  It is easily checked  that if
$\rho_\delta$  defined  as $\delta  \mu  /  \nu$ is  such  that  $\rho_\delta  < 1$,  then
$(Y_1^\delta(t))$   is  an   ergodic  Markov   process  converging   in   distribution to
$Y_1^\delta(\infty)$ and that
\begin{equation} \label{eq:lambda}
\lambda^*(\delta)\stackrel{\text{def.}}{=} \nu \E(Y_1^\delta(\infty)) =
\delta \mu \E(Y_1^\delta(\infty) \vee 1)=\frac{ \delta  \mu}{(1-\rho_\delta)(1 - \log(1-\rho_\delta))}.
\end{equation}
When $\rho_\delta > 1$, then the process $(Y_1^\delta(t))$ converges almost surely and in expectation to
infinity. In the sequel $\lambda^*(1)$ is simply denoted $\lambda^*$. 

\bigskip
\noindent
In  the following it  will be  assumed, see Condition~(C)  below, that  the rate  function~$r$
converges to $1$ as the first coordinate goes to infinity; as will be seen, the special case $r \equiv 1$ then plays a special role, and so before analyzing the stability properties of  $(X^r(t))$, one begins with  an informal discussion when  the rate function~$r$ is identically equal to~$1$.  Since the departure rate from the system is proportional
to the number  of requests/servers in the second  queue, a large number of  servers in the
second  queue gives  a  high departure  rate, irrespectively  of  the state  of the  first
queue. The input rate of new requests  being constant, the real bottleneck with respect to
stability  is therefore  when  the  first queue  is  large.  The  interaction  of the  two
processes $(X_0^{1}(t))$ and $(X_1^{1}(t))$ is expressed through the indicator function of
the set  $\{X_0^1(t) >  0\}$.  The  second queue $(X_1^{1}(t))$  locally behaves  like the
birth-and-death process $(Y_1^1(t))$ as long as  $(X^{1}_0(t))$ is away from $0$.  The two
cases $\rho_1>1$ and $\rho_1<1$ are considered.

If $\rho_1>1$,  i.e., $\mu >  \nu$, the process  $(X_1^{1}(t))$ is a transient  process as
long as the first coordinate is  non-zero. Consequently, departures from the second queue 
occur  faster and  faster. Since,  on the  other hand,  arrivals occur  at a  steady rate,
departures eventually outpace arrivals. The  fact that the  second queue
grows when $(X_0(t))$ is away from $0$ stabilizes the system independently of the value of
$\lambda$, and so the system should be stable for any $\lambda > 0$.

If $\rho_1 < 1$, and as long as $(X_0(t))$  is away from $0$, the coordinate $(X_1^{1}(t))$
locally behaves like the ergodic Markov process $(Y_1^{1}(t))$. Hence if $(X_0^{1}(t))$ is
non-zero  for   long  enough,   the  requests   in  the  first   queue  see   in  average
$\E(Y_1^{1}(\infty) \vee 1)$ servers which work at rate $\mu$. Therefore, the stability
condition for the  first  queue  should be 
\[
\lambda  <  \mu   \E(Y_1^{1}(\infty)  \vee  1)  =\lambda^*
\]
where $\lambda^*=\lambda^*(1)$ is defined by Equation~\eqref{eq:lambda}.  Otherwise if
$\lambda  > \lambda^*$, the system should be unstable. 


\subsection*{Transience and Recurrence Criteria for $(X^r(t))$} 

\begin{proposition}[Coupling]\label{lem:coupling}
If $X^r(0) = Y^1(0) \in \N^2$, there exists a coupling of the processes $(X^r(t))$ and
$(Y^1(t))$  such that the relation
\begin{equation} \label{eq:coupling-1}
X_0^r(t) \geq Y_0^1(t) \text{ and } X_1^r(t) \leq Y_1^1(t),
\end{equation}
holds for all $t\geq 0$ and for any sample path. 

For any $0 \leq \delta \leq 1$, if 
\[ 
\tau_\delta = \inf \{t \geq 0: r(X^r(t)) \leq \delta \} \text{ and } \sigma = \inf \{ t \geq
0: X_0^r(t) = 0 \},
\]
and if $X^r(0) = Y^\delta(0) \in \N^2$ then there exists a coupling of the processes
$(X^r(t))$ and $(Y^\delta(t))$  such that, for any sample path, the relation 
\begin{equation} \label{eq:coupling-2}
X_0^r(t) \leq Y_0^\delta(t) \text{ and } X_1^r(t) \geq Y_1^\delta(t)
\end{equation}
holds for all $t\leq \tau_\delta \wedge \sigma$. 
\end{proposition}
\begin{proof}
Let $X^r(0)=(x_0,x_1)$ and $Y^1(0)=(y,z)$ be such that $x_0\geq y$ and $x_1\leq
z$, one has to prove that the processes $(X^r(t))$ and $(Y^1(t))$ can be constructed such that
Relation~\eqref{eq:coupling-1} holds at the time of the next jump of one of them.  See
Leskel\"a~\cite{Leskela} for the existence of couplings using analytical techniques.

The arrival rates in the first queue are the same for both processes. If $x_1<z$, a departure
from the second queue for $(Y^1(t))$ or $(X^r(t))$  preserves the order relation~\eqref{eq:coupling-1} and if
$x_1=z$, this departure  occurs at the same rate for both processes and thus the
corresponding instant can be chosen at the same (exponential) time. For the departures
from the first to the second queue, the departure rate for
$(X^r(t))$ is $\mu r(x_0,x_1) (x_1\vee 1)\ind{x_0 > 0}\leq \mu (z\vee 1)$ which is the departure rate
for $(Y^1(t))$, hence the corresponding departure instants can be taken in the
reverse order so that Relation~\eqref{eq:coupling-1} also holds at the next jump instant.  The
first part of the proposition is proved. 

The   rest  of   the   proof  is   done   in  a   similar   way:  The initial states  $X^r(0)=(x_0,x_1)$   and
$Y^\delta(0)=(y,z)$ are such that $x_0\leq y$ and $x_1\geq z$. With the killing of
the processes at time $\tau_\delta \wedge \sigma$ one can assume additionally that $x_0\not=0$
and that the relation $r(x_0,x_1)\geq \delta$ holds; Under these assumptions one can
check by inspecting the next transition that~\eqref{eq:coupling-2} holds. The proposition
is proved. 
\end{proof}
\begin{proposition}\label{prop:transience}
Under the condition $\mu < \nu$,  the relation
\[
\liminf_{t \to +\infty} \frac{X_0^r(t) }{t} \geq \lambda - \lambda^*
\]
holds  almost surely. 
In particular, if $\mu < \nu$ and $\lambda > \lambda^*$, then the process $(X^r(t))$ is transient.
\end{proposition}
\begin{proof}
By Proposition~\ref{lem:coupling}, one can assume that there exists a version of
$(Y^1(t))$ such that $X^r_0(0)=Y_0^1(0)$ and the relation $X^r_0(t)\geq Y_0^1(t)$ holds
for any $t\geq 0$. 
From  Definition~\eqref{defQd} of the $Q$-matrix of $(Y^1(t))$, one has, for $t\geq 0$,
\[
Y_0^1(t)=Y_0^1(0)+ {\cal N}_\lambda(t) -A(t),
\]
where $({\cal N}_\lambda(t))$  is a Poisson process with parameter $\lambda$ and $(A(t))$
is  the number of arrivals (jumps of size $1$) for the second coordinate $(Y_1^1(t))$: in
particular
\[
\E(A(t))= \mu \E\left(\int_0^t Y_1^1(s)\vee 1\,ds\right).
\]
Since $(Y_1^1(t))$ is an ergodic Markov process under the condition $\mu<\nu$, the
ergodic theorem in this setting gives that
\[
\lim_{t\to+\infty} \frac{1}{t}A(t)=\lim_{t\to+\infty} \frac{1}{t}\E(A(t))=\mu \E\left(Y_1^1(\infty)\vee 1\right)=\lambda^*,
\]
by Equation~\eqref{eq:lambda}, hence $(Y_0^1(t)/t)$ converges almost surely to
$\lambda-\lambda^*$. The proposition is proved. 
\end{proof}
\noindent
The next result establishes the ergodicity result of this section. 
\begin{proposition} \label{prop:ergodicity}
If the rate function $r$ is such that, for any $x_1\in\N$,
\begin{equation}
\lim_{x_0 \to +\infty} r(x_0,x_1) = 1, \tag{C}
\end{equation}
and if $\mu \geq \nu$, or if $\mu < \nu$ and $\lambda < \lambda^*$ with 
\begin{equation} \label{eq:lambda2}
\lambda^*=\frac{  \mu}{(1-\rho)(1 - \log(1-\rho))},
\end{equation}
and $\rho=\mu/\nu$,  then $(X^r(t))$ is an
ergodic Markov process.
\end{proposition}
\noindent 
Note that Condition~(C) is satisfied for the functions $r$ considered in  the models
considered by N\'u\~nez-Queija and Prabhu~ \cite{Queija08:0} and  in Susitaival
\etal~\cite{Susitaival06:0}. See above. 
\begin{proof}
If $x=(x_0,x_1)\in\R^2$, $|x|$ denotes the norm of $x$, $|x|=|x_0|+|x_1|$. 
The proof uses Foster's criterion as stated in Robert~\cite[Theorem~9.7]{Robert03}. If
there exist  constants $K_0$, $K_1$, $t_0$ and $t_1$ such that
\begin{align}
\sup_{x_1 \geq K_1} & \E_{(x_0,x_1)}(|X^r(t_1)| - |x|) < 0,\label{eq:foster-x2}\\
\sup_{x_1 < K_1, x_2 \geq K_2} & \E_{(x_0,x_1)}(|X^r(t_0)| - |x|) < 0,\label{eq:foster-x1}
\end{align}
then the Markov process $(X^r(t))$ is ergodic. 

Relation~\eqref{eq:foster-x2} is straightforward to establish: if $x_1 \geq K_1$, one gets, by
considering only $K_1$ of the $x_1$ initial servers in the second queue and the Poisson arrivals, that
\[ 
\E_{(x_0,x_1)}(|X^r(1)| - |x|) \leq \lambda - K_1 (1 - e^{-\nu}), 
\]
hence it is enough to take $t_1 = 1$ and $K_1 = (\lambda + 1) / (1-e^{-\nu})$ to have
Relation~\eqref{eq:foster-x2}. 

One has therefore to establish Inequality~\eqref{eq:foster-x1}. 
Let $\tau_\delta$ and $\sigma$  be the stopping times introduced in
Proposition~\ref{lem:coupling}, one first proves an intermediate result: for any $t>0$ and any
$x_1 \in \N$, 
\begin{equation}\label{Jacqueson}
\lim_{x_0\to+\infty} \P_{(x_0,x_1)}(\sigma\wedge \tau_\delta \leq t)=0.
\end{equation}
Fix $x_1 \in \N$ and $t \geq 0$: 
for $\eps>0$, there exists $D_1$ such that 
\[
\P_{x_1}\left(\sup_{0\leq s\leq t} Y_1^1(s)\geq D_1\right)\leq \eps,
\]
from Proposition~\ref{lem:coupling}, this gives the relation valid for all $  x_0\geq 0$,
\[
\P_{(x_0,x_1)}\left(\sup_{0\leq s\leq t} X_1^r(s)\geq D_1\right)\leq \eps.
\]
By  Condition~(C), there exists $\gamma \geq 0$ (that depends on $x_1$) such that
$ r(x_0,x_1) \geq \delta$ when $x_0 \geq \gamma$.
As long as $(X^r(t))$ stays in  the subset $\{(y_0,y_1): y_1\leq D_1\}$,  the transition
rates of the first component $(X_0^r(t))$ are uniformly bounded. Consequently, there
exists $K$ such that, for $x_0\geq K$, 
\[
\P_{(x_0,x_1)}\left[\sup_{s\leq t} X_0^r(s) \le \gamma,\ \sup_{s\leq t}X_1^r(s)\leq D_1, \right]\leq \eps.
\]
Relation~\eqref{Jacqueson} follows from the last two inequalities and the
identity
\[
\P_{(x_0,x_1)}(\sigma\wedge \tau_\delta \leq t)\leq \P_{(x_0,x_1)}\left(\sup_{s\leq t} X_0^r(s) \le \gamma\right).
\]

One returns to the proof of Inequality~\eqref{eq:foster-x1}.
By definition of the $Q$-matrix of the process $(X^r(t))$, 
\[ 
\E_{(x_0,x_1)}(|X^r(t|) - |x|) = \lambda t - \nu \int_0^t \E_{(x_0,x_1)}(X^r_1(u)) du,\ x \in \N^2, \ t \geq 0. 
\]
For any $x \in \N^2$,   there
exists a version of $(Y^\delta(t))$  with initial condition $Y^\delta(0)=X^r(0)=x$, and
such that Relation~\eqref{eq:coupling-2} holds for $t<\tau_\delta \wedge \sigma$, in
particular 
\begin{multline*} 
\E_x(X^r_1(t)) \geq \E_x(X^r_1(t); t< \tau_\delta \wedge \sigma)\\
\geq \E_x(Y_1^\delta(t); t < \tau_\delta \wedge \sigma) = \E_x(Y_1^\delta(t)) - \E_x(Y_1^\delta(t); t\geq \tau_\delta \wedge \sigma).
\end{multline*}
Cauchy-Schwarz inequality shows that for any $t \geq 0$ and $x \in \N^2$
\[
		\begin{split}
			\int_0^t \E_x(Y_1^\delta(u); \tau_\delta \wedge \sigma \leq u)\, du & \leq \int_0^t \sqrt{\E_{x}\left[\left(Y_1^\delta(u)\right)^2 \right]} \sqrt{\P_x(\tau_\delta \wedge \sigma\leq u)}\, du\\
			& \leq \sqrt{\P_x(\tau_\delta \wedge \sigma\leq t)} \int_0^t \sqrt{\E_{x}\left[\left(Y_1^\delta(u)\right)^2 \right]} \,du,
		\end{split}
\]
by gathering these inequalities,  and by using the fact that the process $(Y^\delta_1(t))$
depends only on $x_1$ and not $x_0$, one finally gets the relation 
\begin{equation} \label{eq:1}
\frac{1}{t} \E_x(|X(t)| - |x|)
\leq \lambda - \frac{\nu}{t} \int_0^t \E_{x_1}(Y_1^\delta(u)) \,du + c(x_1,t)\sqrt{\P_x(\tau_\delta \wedge \sigma \leq t)}
\end{equation}
with
\[ c(x_1, t) = \frac{\nu}{t}\int_0^t \sqrt{\E_{x_1}\left[\left(Y_1^\delta(u)\right)^2 \right]} du. \]

\noindent
Two cases are considered.
\begin{enumerate}
\item If $\mu > \nu$, if $\delta<1$ is such that $ \delta \mu > \nu$, the process $(Y_1^\delta(t))$
is transient, so that 
\[
\lim_{t\to+\infty} \frac{1}{ t} \int_0^t \E_{x_1}(Y_1^\delta(u))\, du=+\infty,
\]
for each $x_1 \geq 0$. 

\item If $\mu < \nu$, one takes $\delta = 1$, or if $\mu = \nu$, one takes $\delta < 1$ close enough to $1$ so that $\lambda < \lambda^*(\delta)$. In both cases, $\lambda < \lambda^*(\delta)$ and the process $(Y_1^\delta(t))$ converges in
distribution, hence
\[
\lim_{t\to +\infty}\frac{\nu}{t} \int_0^t \E_{x_1}(Y_1^\delta(u))\, du=\nu \E \left(Y_1^\delta(\infty)\right) = \lambda^*(\delta) > \lambda
\]
for each $x_1 \geq 0$. 
\end{enumerate}
Consequently in both cases,  there exist constants $\eta > 0$, $\delta<1$ and $t_0 > 0$ such that for any $x_1 \leq
K_1$, 
\begin{equation}\label{eq2}
\lambda - \frac{\nu}{t_0} \int_0^{t_0} \E_{x_1}(Y_1^\delta(u)) du \leq -\eta,
\end{equation}
with Relation~\eqref{eq:1}, one gets that if $x_1 \leq K_1$ then 
\[ 
\frac{1}{t_0} \E_x(|X(t_0)| - |x|) \leq -\eta + c^* \sqrt{\P_x(\tau_\delta \wedge \sigma  \leq t_0)}, 
\]
where  $c^* =  \max(c(n, t_0), 0 \leq n \leq K_1)$. By Identity~\eqref{Jacqueson}, there
exists $K_0$ such that, for all $x_0\geq K_0$ and $x_1\leq K_1$,  the relation
\[
c^*\sqrt{\P_{(x_0,x_1)}(\tau_\delta \wedge \sigma \leq t_0)}\leq \frac{\eta}{2}
\]
holds. This  relation and the inequalities~\eqref{eq2} and~\eqref{eq:1} give Inequality~\eqref{eq:foster-x1}.
The proposition is proved. 
\end{proof}

\subsection*{Another Boundary Condition}
The boundary  condition $x_1 \vee 1$  in the transition  rates of $(X(t))$,
Equation~\eqref{def},  prevents  the second coordinate from ending  up in the absorbing state
$0$. It amounts to suppose  that a permanent 
server gets activated when  no node  may offer the file. Another  way to avoid this
absorbing state  is to  suppose that a  permanent node  is always active,  which gives
transition rates with $x_1 + 1$ instead. This choice was for instance made
in N\'u\~nez-Queija and Prabhu~ ~\cite{Queija08:0}. All our  results apply for this other
boundary condition: the only difference that is when $\nu > \mu$,  the  value of  the
threshold $\lambda^*$ of Equation~\eqref{eq:lambda} 
is given by the quantity $\lambda^* = {\mu \nu}/{(\nu - \mu)}$.
\section{The Two-Chunk Network} \label{sec:net}
In this section it is assumed that a file of two chunks is distributed by the file-sharing
system corresponding  to Figure~\ref{chunk-fig}.  Chunks  are delivered in  the sequential
order.  The analysis  of this simple file-sharing system gives a  hint of the difficulties
one can encounter when dealing with multiple chunks, in particular with nodes distributing
{\em and} receiving chunks.

For  $k = 0, 1$ and  $t\geq  0$, the variable $X_{k}(t)$  denotes the  number
of  requests downloading the $(k{+}1)$st chunk; for $k = 2$, $X_{2}(t)$
is the number of requests  having all the chunks. When taking into account the boundaries
in the transition rates described in Figure~\ref{chunk-fig}, one gets the following
$Q$-matrix for the three-dimensional Markov process  $(X(t)) = (X_k(t), k = 0,1,2)$
\begin{multline*}
Q(f)(x)=\lambda[f(x+e_0)-f(x)]+ \mu_1(x_1\vee 1)[f(x+e_1-e_{0})-f(x)]\ind{x_{0}>0}\\+
\mu_2(x_2\vee 1)[f(x+e_2-e_{1})-f(x)]\ind{x_{1}>0}+\nu x_{2}[f(x-e_{2})-f(x)],
\end{multline*}
where $x\in\N^{3}$,  $f:\N^{3}\to\R_+$ is a  function and for,  $0\leq k\leq 2$,  $e_k \in
\N^{3}$ is the $k$th  unit vector. Note that, as before, to  avoid absorbing states, it is
assumed that there is a server for the $k$th chunk when $x_{k}=0$.

The stability behavior of $(X(t))$ depends on the values of the parameters $\mu_1$, $\mu_2$ and $\nu$. Three cases need to be distinguished: $(1)$~$\mu_1 > \mu_2 - \nu > 0$, $(2)$~$\mu_2 - \nu > \mu_1$ and $(3)$~$\mu_2 < \nu$. 
In each case, the method of proof is similar to the one used in the proof of
Proposition~\ref{prop:ergodicity} and relies on Foster's criterion as stated in
Robert~\cite[Theorem 9.7]{Robert03}. One investigates the case when the size $x_0+x_1+x_2$
of the initial state $X(0) = x = (x_0, x_1, x_2)$ is large. We briefly discuss the three
situations which will be  analyzed in the following. 

\begin{enumerate}
\item If $x_2$ is large, then the total number of customers decreases instantaneously.

\item If $x_1$ is large and $x_2$ is small, then the last queue $(X_2(t))$ behaves like the birth-and-death process $(Z(t))$ of Section~\ref{sec:Interacting}, which is transient if $\mu_2 > \nu$ and stable otherwise.

\item If $x_0$ is large and $x_1$ and $x_2$ are small, the two last queues $(X_1(t), X_2(t))$ behave like the system $(X^S(t)) = (X^S_1(t), X_2^S(t))$, defined by its generator
\begin{multline} \label{eq:saturated}
Q^S(f)(z)=\mu_{1}(z_{1}\vee 1)[f(z+e_{1})-f(z)]\\+  \mu_{2} (z_{2}\vee 1) 
[f(z+e_2-e_{1})-f(z)]\ind{z_{1}>0} +\nu z_2 [f(z-e_2)-f(z)],
\end{multline}
where  $z\in\N^{2}$   and  $f:\N^{2}\to\R_+$  is   an arbitrary  function. 
\end{enumerate}

Since $(X^S(t))$ corresponds to the two last  queues when the first queue is saturated, we
call it the \emph{saturated system}: it plays
a  crucial  for the  stability  of  $(X(t))$, in  particular  the  asymptotic behavior  of
$\E(X_2^S(t))$ determines the output of the  system when the first queue is saturated. The
properties of $(X^S(t))$ that will be needed are gathered in the next proposition.

\begin{proposition}[Saturated system]\label{prop:saturated}
	Let $x \in \N^2$.
	\begin{enumerate}
		\item If $\mu_1 > \mu_2 - \nu > 0$, then $\E_x(X_2^S(t)) \to +\infty$ as $t \to +\infty$.
		\item If $\mu_2 - \nu > \mu_1$, then $(X^S(t))$ is positive recurrent.
		\item If $\nu > \mu_2$, then $\nu \E_x(X_2^S(t)) \to \lambda^*$ as $t \to +\infty$, with $\lambda^*$ given by
		\begin{equation} \label{eq:lambda-2}
			\lambda^* = \frac{\mu_2}{(1-\mu_2/\nu) (1 - \log(1 - \mu_2/\nu))}.
		\end{equation}
	\end{enumerate}
\end{proposition}

\begin{proof}
	The crucial observation is that as long as $X_1^S(t) > 0$, the process $(X^S(t))$ can be coupled with the process $(W(t), Z(t))$ of the last part of Section~\ref{sec:Interacting}, i.e., $(X_2^S(t))$ behaves like a birth-and-death process and $(X_1^S(t))$ is a Yule process with particles killed at the instants of birth of $(X_2^S(t))$.
	
	More formally, in the sequel let $(Z(t))$ be the process with $Q$-matrix defined by Relation~\eqref{eqZ} with $\mu_Z = \mu_{2}$, and define $(\sigma_n)$ its sequence of birth times. Let $(W(t))$ be a Yule process with parameter $\mu_1$ killed at $(\sigma_n)$. Then $(X^S(t))$ and $(W(t), Z(t))$ can be coupled so that
	\[ (X_1^S(t), X_2^S(t), 0 \leq t \leq T) = (W(t), Z(t), 0 \leq t \leq T) \]
	with $T = \inf\{ t \geq 0: X_1^S(t) = 0 \}$.  We give a full proof of the two first cases, the last case uses similar techniques and so we leave the details to the reader.
	%
	
\bigskip
	\noindent\textit{Proof of case~$(1)$.} 
	Since $\mu_2 > \nu$, the process $(e^{-(\mu_{2}{-}\nu)t}Z(t))$ converges almost surely to a finite and positive random variable
	$M_{Z}(\infty)$ by Corollary~\ref{cor:Z(infty)}. Moreover, since $\mu_{1} > \mu_{2} - \nu > 0$, Corollary~\ref{cor:series} entails that
	\[
	\sum_{n \geq 1} e^{-\mu_{1} \sigma_n} < +\infty,\quad \text{almost surely.}
	\]
	By Proposition~\ref{KillProp}, this shows that $(W(t))$ survives with positive probability. In this event, say ${\cal E}$, it holds that $W(t) \geq 1$ at all times and
	\[
	\lim_{t\to +\infty} Z(t) = +\infty.
	\]
	In the event ${\cal E}$, $W(t) \geq 1$ and therefore $(X_1^S(t), X_2^S(t)) = (W(t), Z(t))$: in particular, we have exhibited an event of positive probability where $X_2^S(t) \to +\infty$, which proves the claim. Note that we implicitly assumed that no coordinate of the initial state is zero, these cases being dealt with easily. 	

\bigskip
	\noindent \textit{Proof of case~$(2)$.} 
	This is the most delicate case. 
	By Proposition~\ref{prop:Z(H0)} and since $\mu_2 - \nu > \mu_1$, one has that $(X^S_{1}(t))$ returns infinitely often to
	$0$.  When  $(X^S_{1}(t))$ is  at $0$  it jumps to $1$ after  an exponential  time with
	parameter $\mu_1$. One defines the sequences $(H_{0,n})$, $(E_{\mu,n})$ and $(S_n)$ by induction: 
	%
	$S_{0} = 0$ by convention, and for $n \geq 1$, $S_n = \sum_{k=1}^n (H_{0,k} + E_{\mu, k})$,
	\[ H_{0,n} = \inf\{t \geq S_{n-1}: X_1^S(t) = 0\} \text{ and } E_{\mu, n} = \inf\{ t \geq S_{n-1} + H_{0,n}: X_1^S(t) = 1 \}. \]
	The sequence $(E_{\mu,n})$ represents the successive times at $0$, $(S_n)$ the times at which $(X_1^S(t))$ hits $0$, and $(H_{0,n})$ the times needed for $(X_1^S(t))$ to hit $0$. On the time intervals $[S_n, S_n + H_{0,n+1}]$, $(X^S(t))$ can be coupled with a version of $(W(t), Z(t))$. This shows that, with the notations of Proposition~\ref{prop:Z(H0)}, $X_2^S(S_n)$ is equal in distribution to $Z(H_{0})$ with the initial condition $(W(0), Z(0)) = (X_1^S(S_{n-1}), X_2^S(S_{n-1}))$. Then while $X_{1}^S$ is at $0$, the dynamics of $X_{2}^S$ is simple, since it just empties. These two remarks show that $X_2(S_{n+1})$ is equal in distribution to
	\[
	\left( X_{2}^S(S_{n+1}) \, | \, X^S(S_n) = x \right) \stackrel{\text{dist.}}{=} \left( \sum_{i=1}^{Z(H_{0})} \ind{E_{\nu,i}> E_{\mu_1}} \, \big | \, (W(0), Z(0)) = x \right),
	\]
	with $(E_{\nu,i})$ i.i.d., exponentially distributed with parameter $\nu$ and independent of $E_{\mu_1}$. For $i \geq 1$, $E_{\nu, i}$ represents the residence time of the $i$th customer present in the second queue at time $S_n + H_{0,n+1}$. 
	
	Consequently, and since $X_1^S(S_n) = 1$ for $n \geq 1$, $(X_2^S(S_n), n \geq 1)$ is a Markov chain whose transitions are defined by Relation~\eqref{eqV} with $p = \nu / (\nu + \mu_1)$; note that $(X_2^S(S_n), n \geq 0)$ has the same dynamics only when $X_{1}^S(0) = 1$. In the sequel define for any $K > 0$ the stopping $T_K$:
	%
	\[
	T_K = \inf\{n \geq 0: X_2^S(S_n) \leq K, X_1^S(S_n) = 1 \}.
	\]
	The ergodicity of $(X^S(t))$ will follow from the finiteness of $\E_{(x_1,x_2)}(T_K)$ for some~$K$ large enough and for arbitrary $x = (x_1, x_2) \in \N^2$: indeed, $T_K$ is greater than the time needed to return to the finite set $\{ x_1, x_2: x_1 = 1, x_2 \leq K \}$. The strong Markov property of $(X^S(t))$ applied at time $S_1$ gives
	\[ \E_{(x_1, x_2)}(T_K) \leq 2\E_{(x_1, x_2)}(S_1) + \E_{(x_1, x_2)} \left[ \E_{(1,X_{2}^S(S_1))}(T_K) \right], \]
	and so one only needs to study $T_K$ conditioned on $\{X_{1}^S(0) = 1\}$ since $\E_{(x_1, x_2)}(S_1)$ is finite in view of Proposition~\ref{prop:Z(H0)}. Thus one can assume that $(X_2^S(S_n), n \geq 0)$ is a time-homogeneous Markov chain whose transitions are defined by Relation~\eqref{eqV}, and 
%
	with $N_K$ defined in Proposition~\ref{prop:V}, the identity
	\begin{equation}\label{Dauv}
	T_K = \sum_{i=1}^{N_K} (H_{0,i}+E_{\mu_1,i})
	\end{equation}
	holds. For $i \geq 0$, the Markov property of $(X_2^S(S_n), n \geq 0)$ gives for any $x_2 \geq 0$
	\[
	\E_{(1, x_2)}\left(H_{0,i}\ind{i \leq N_K}\right)=\E_{(1, x_2)}\left(\E_{(1,X_2^S(S_i))}\left(H_{0}\right)\ind{i \leq N_K}\right).
	\]
	With the same argument as in the proof of Proposition~\ref{prop:V}, one has
	\[
	\E_{(1,X_2^S(S_i))}(H_0) \leq \E_{(1,0)}(H_0) < +\infty,
	\]
	with Equations~\eqref{Dauv} and~\eqref{eq:N} of Proposition~\ref{prop:V}, one gets that for some $\gamma > 0$ and some $K > 0$,
	\[
	\E_{(x_1, x_2)}(T_K) \leq 2 \E_{(x_1, x_2)}(S_1) + C \E_{(x_1, x_2)}\left[\log\left(1+X_{2}^S(S_1)\right)\right]
	\]
	with the constant $C = (\E_{(1,0)} (H_0)+1 / \mu_2) / \gamma$. This last term is finite for
	any $(x_1, x_2)$ in view of Proposition~\ref{prop:Z(H0)}, which proves the proposition. 
	
\bigskip
	\noindent\textit{Sketch of proof of case~$(3)$.} Since $\mu_1 > 0 > \mu_2 - \nu$, Corollary~\ref{cor:series} shows that $\sum_{n \geq 1} e^{-\mu_1 \sigma_n}$ is finite almost surely, and Proposition~\ref{KillProp} implies that $(X_1^S(t))$ survives with positive probability. Using the fact that the process $(Z(t))$ is stable since $\mu_2 < \nu$, together with similar arguments as before, one can show that the process $(X_1^S(t))$ actually survives almost surely. Thus after some time $(X_2^S(t))$ can be coupled forever with $(Z(t))$, and $\lambda^*$ is precisely the value of $\nu \E(Z(\infty))$ in stationary regime.
\end{proof}

\begin{proposition}
	Let $(X(t)) = (X_0(t), X_1(t), X_2(t))$ be the Markov process describing the linear file-sharing network with parameters $\lambda, \mu_1, \mu_2$ and $\nu$.
	\begin{enumerate}
		\item If $\mu_1 > \mu_2 - \nu > 0$, then $(X(t))$ is ergodic for any $\lambda > 0$.
		\item If $\mu_2 - \nu > \mu_1$, let $\lambda^S$ be the value of the expectation of $\nu X_2^S$ at equilibrium, which exists by Proposition~\ref{prop:saturated}. Then $(X(t))$ is ergodic for $\lambda < \lambda^S$ and transient for $\lambda > \lambda^S$.
		\item If $\nu > \mu_2$, then $(X(t))$ is ergodic for $\lambda < \lambda^*$ and transient for $\lambda > \lambda^*$, with $\lambda^*$ given by~\eqref{eq:lambda-2}.
	\end{enumerate}
\end{proposition}

In the second case, preliminary investigations seem to suggest that $\lambda^S = +\infty$ is possible. In such situations, the stability region would actually be infinite. The intuitive reason is that $(X_2^S(t))$ grows exponentially fast in periods where $X_1^S(t) > 0$, and for a duration
$H_0$. Thus when $(X_1^S(t))$ hits $0$, $X_2^S(H_0) \approx e^{(\mu_2 - \nu) H_0}$, and although some exponential moments of $H_0$ are finite by~\eqref{eq:exp(H_0)}, it seems plausible that $\E(e^{(\mu_2-\nu)H_0}) = +\infty$ sometimes.

\begin{proof}
	In all three cases the proof relies on Foster's criterion as stated in
	Theorem~$9.7$ of Robert~\cite{Robert03}. Let $X(0)=x=(x_0, x_1, x_2)\in\N^{3}$, assume that $\lVert x \rVert = x_0 + x_1 + x_2$ is large. First we inspect the case when $x_2$ is large, it is common to all the cases. 

	Since the last queue serves at rate $\nu$ each request, for $t\geq 0$,
	\[
	\E_x(\|X(t)\|)\leq \|x\|+ \lambda t- x_2\left(1-e^{-\nu t}\right).
	\]
	Define $t_2 = 1$ and let $K_2$ be such that $\lambda-K_2(1-e^{-\nu})\leq -1$, then the relation
	\[
	\E_x(\|X(t_2)\|)- \|x\|\leq -1,
	\]
	holds when $x_2 \geq K_2$, for all $x_0, x_1 \geq 0$. In the sequel the constants $t_2$ and $K_2$ are fixed, and the process $(Z(t))$ is defined similarly as in the proof of Proposition~\ref{prop:saturated}. We give a full proof of case~$(1)$, the other cases use similar techniques and so we leave the details to the reader.

\bigskip
	\noindent\textit{Proof of case~$(1)$.} Assume that $\mu_1 > \mu_2 - \nu > 0$: in particular $(Z(t))$ is transient, thus there exists some $t_{1}$ such that for any $x_2 < K_2$,
	\[
	\nu \int_0^{t_{1}}\E_{x_2}\left(Z(u)\right)\,du\geq \lambda t_{1} + 2.
	\]
	The two processes $(Z(t))$ and $(X(t))$ can be built on the same probability space such that if they start from the same initial state, then 
	the two processes $(Z(t))$ and $(X_2(t))$ are
	identical as long as $X_{1}(t)$ stays positive. Since moreover the hitting time $\inf\{t
	\geq 0: X_{1}(t) = 0 \}$ goes to infinity as $x_{1}$ goes to infinity for any $x_2 < K_2$, one gets that there exists $K_{1}$ such that 
	if $x_{1} \geq K_{1}$ and $x_2 < K_2$, then the relation 
	\begin{align*}
	\E_x(\|X(t_{1})\|)- \|x\| & = \lambda t_{1} - \nu \int_0^{t_{1}} \E_{x}(X_2(u))\, du\\
	& \leq \lambda t_{1} - \left( \nu\int_0^{t_{1}} \E_{x_2} \left(Z(u)\right)\,du - 1\right) \leq -1
	\end{align*}
	holds.
	
	In a similar way, by saturating the first queue and coupling the two last queues $(X_1(t), X_2(t))$ with the saturated system $(X^S(t))$, one gets in view of case~$(1)$ of Proposition~\ref{prop:saturated} that there exist constants $t_{0}$ and $K_{0}$ such that  if $x_2 < K_2$,
	and $x_{1} < K_{1}$ and $x_{0} \geq K_{0}$, then  
	\[
	\E_x(\|X(t_{0})\|)- \|x\|\leq  -1,
	\]
	which ends the proof of this case.

\bigskip
	\noindent\textit{Sketch of proof of case~$(2)$.} Since $\mu_2 > \nu$, the process $(Z(t))$ is still transient and one can choose the same constants $t_1, K_1$ as in the previous case to deal with the case $x_1 \geq K_1, x_2 < K_2$. When $x_0 \geq K_0$ and $x_1 < K_1, x_2 < K_2$, then one can justify for $K_0$ and $t_0$ large enough the approximation
	\[ \E_x(\lVert X(t_0) \rVert) - \lVert x \rVert \approx \lambda t_0 - \nu \int_0^{t_0} \E_{(x_1, x_2)}(X_2^S(u)) \, du \approx (\lambda - \lambda^S) t_0. \]
	Hence for $\lambda < \lambda^S$, the stability result follows. For $\lambda > \lambda^S$, this suggests that the first coordinated $(X_0(t))$ builds up, and similarly as in the proof of Proposition~\ref{prop:transience} one can show that indeed in this case
	\[ \liminf_{t \to +\infty} \frac{X_0(t)}{t} \geq \lambda - \lambda^S > 0,\ \text{ almost surely}.  \]

\bigskip	
	\noindent\textit{Sketch of proof of case~$(3)$.} Since $\nu > \mu_2$, the process $(Z(t))$ is stable, and it acts as a bottleneck on the system. When $x_0 < K_0$ and $x_1 \geq K_1$, then one can as before justify the approximation, for $K_1$ and $t_1$ large enough,
	\[ \E_x(\lVert X(t_1) \rVert) - \lVert x \rVert \approx \lambda t_1 - \nu \int_0^{t_1} \E_{x_2}(Z(u)) \, du \approx (\lambda - \lambda^*)t_1. \]
	Similarly, the third case of Proposition~\ref{prop:saturated} shows that when $x_0 \geq K_0$ and $x_1 < K_1, x_2 < K_2$, again for $K_0$ and $t_0$ large enough,
	\[ \E_x(\lVert X(t_0) \rVert) - \lVert x \rVert \approx \lambda t_0 - \nu \int_0^{t_0} \E_{(x_1, x_2)}(X_2^S(u)) \, du \approx (\lambda - \lambda^*) t_0. \]
	This shows that the system is ergodic for $\lambda < \lambda^*$. The transience in the other case $\lambda > \lambda^*$ follows easily from the fact that the last queue $(X_2(t))$ is stochastically dominated by $(Z(t))$, thus $\lambda^*$ is the maximal output rate of the system.
\end{proof}

\appendix

\section{Proof of Proposition~\protect\ref{lemma:births}}
In this appendix the notations of Section~\ref{sec:Interacting} are used.
Since the random variable $(B_\sigma(t)\mid Z(0)=0)$ is stochastically smaller than $(B_\sigma(t)\mid
Z(0)=z)$ for any $z \in \N$, it is enough to show that for $\eta < \eta^*(\nu / \mu_Z)$ 
\[
\E_0 \left[ \sup_{t \geq \sigma_1} \left(e^{\eta\alpha t} B_\sigma(t)^{-\eta} \right) \right] < +\infty,
\]
where $\alpha = \mu_Z - \nu > 0$. 

Note that the process $(B_\sigma(t + \sigma_1), t \geq 0)$ under $\P_0$ has the same 
distribution as $(B_\sigma(t)+1, t \geq 0)$ under $\P_1$, and by independence of
$\sigma_1$, an exponentially random variable with  parameter $\mu_Z$, 
and $(B_\sigma(t+\sigma_1), t \geq 0)$, one gets  
\[
\E_0 \left[ \sup_{t \geq \sigma_1} \left(e^{\eta\alpha t} B_\sigma(t)^{-\eta} \right) \right] = \E_0 \left( e^{\eta \alpha \sigma_1} \right) \E_1 \left[ \sup_{t \geq 0} \left(e^{\eta\alpha t} \left(B_\sigma(t)+1\right)^{-\eta} \right) \right].
\]
Since $\alpha < \mu_Z$ and $\eta^*(\nu / \mu_Z)<1$, then $\E_0 \left( \exp(\eta \alpha \sigma_1)
\right)$ is finite, and all one needs to prove is that the second term is finite as
well. 

Define $\tau$ as the last time $Z(t) = 0$: 
\[ \tau = \sup\{t \geq 0: Z(t) = 0\}, \]
with the convention that $\tau = +\infty$ if $(Z(t))$ never returns to $0$. Recall that,
because of the assumption $\mu_Z>\nu$, with probability $1$, the process $(Z(t))$ returns
to $0$ a finite number of times. 

Conditioned on the event  $\{\tau =  +\infty\}$, the  process $(Z(t))$  is  a $(p,\lambda)$-branching
process conditioned on survival, with $\lambda  = \mu_Z+\nu$ and $p = \mu_Z/\lambda$. Such
a branching  process conditioned on survival can  be decomposed as $Z=Z_{(1)}  + Y$, where
$(Y(t))$  is   a  Yule  process  $(Y(t))$   with  parameter  $\alpha$.   See  Athreya  and
Ney~\cite{Athreya72:0}.  Consequently, for any $0 < \eta < 1$,
\[
\E_1 \left[ \sup_{t \geq 0} \left( e^{\eta\alpha t} \left(B_\sigma(t)+1\right)^{-\eta} \right) | \, \tau = +\infty \right] 
\leq \E_1 \left[ \sup_{t \geq 0} \left( e^{\eta\alpha t} Y(t)^{-\eta} \right) \right].
\]
Since the $n$th split time $t_n$ of $(Y(t))$ is distributed like the maximum of $n$
i.i.d.\ exponential random variables, $Y(t)$ for $t \geq 0$ is geometrically distributed
with parameter $1-e^{-\alpha t}$, hence,
\begin{align*}
\sup_{t\geq 0}\left[e^{\eta \alpha t} \E_1 \left( \frac{1}{Y(t)^{\eta}} \right)\right] &= 
\sup_{t\geq 0} \left[ e^{-(1- \eta) \alpha t}
\sum_{k \geq 1}\frac{(1-e^{-\alpha t})^{k-1}}{k^{\eta}}\right]
\\&\leq \sup_{0\leq u\leq 1}\left[(1-u)^{1-\eta} \sum_{k \geq 1}\frac{u^{k-1}}{k^{\eta}} \right].
\end{align*}
For $0< u< 1$, the relation
\begin{align*}
(1-u)^{1-\eta} \sum_{k \geq 1}\frac{u^{k-1}}{k^{\eta}}
&\leq (1-u)^{1-\eta} \int_0^\infty \frac{u^{x}}{(1+x)^{\eta}}\, dx,\\
&= \left(\frac{1-u}{-\log u}\right)^{1-\eta} \int_0^\infty \frac{e^{-x}}{(x-\log u)^{\eta}}\, dx,
\end{align*}
holds, hence
\[
\sup_{t\geq 0}\left[e^{\eta \alpha t} \E_1 \left( \frac{1}{Y(t)^{\eta}} \right)\right] <+\infty.
\]
The process $(e^{-\alpha t} Y(t))$ being a martingale,  by convexity the process
$(e^{\eta\alpha t} Y(t)^{-\eta})$ is a non-negative sub-martingale. For any $\eta\in(0,1)$ Doob's
$L_p$ inequality gives the existence of a finite $q(\eta) > 0$ such that 
\[
\E_1 \left[ \sup_{t \geq 0} \Big( e^{\eta \alpha t} Y(t)^{-\eta} \Big) \right] \leq q(\eta)  \sup_{t \geq 0} \left[ e^{\eta \alpha t} \E_1 \left( \frac{1}{Y(t)^{\eta}} \right) \right]<+\infty.
\]
The following result has therefore been proved.
\begin{lemma} \label{lemma:infinity}
For any $0 < \eta < 1$,
	\[ \E_1\left. \left[ \sup_{t \geq 0} \left( e^{\eta\alpha t} \left(B_\sigma(t) + 1\right)^{-\eta} \right) \right|\, \tau = +\infty \right] < +\infty. \]
\end{lemma}
On the event $\{\tau < +\infty\}$, $(Z(t))$ hits a geometric number of times $0$ and
then couples with a $(p, \lambda)$-branching process conditioned on survival. On this event,
\begin{multline*}
\sup_{t \geq 0} \left( e^{\eta\alpha t} \left(B_\sigma(t) + 1\right)^{-\eta} \right)\\
= \max \left( \sup_{0 \leq t \leq \tau} \left( e^{\eta\alpha t} \left(B_\sigma(t) + 1\right)^{-\eta} \right), \sup_{t \geq \tau} \left( e^{\eta\alpha t} \left(B_\sigma(t) + 1\right)^{-\eta} \right) \right)\\
\leq e^{\eta\alpha \tau} \left( 1 + \sup_{t \geq 0} \left( e^{\eta\alpha t} \left(B'_\sigma(t) + 1\right)^{-\eta} \right) \right)
\end{multline*}
where $B_\sigma'(t)$ for $t \geq \tau$ is the number of births in $(\tau, t]$ of a $(p, \lambda)$-branching process
conditioned on survival and independent of the variable $\tau$, consequently
\begin{multline*}
\E_1 \left[ \left.\sup_{t \geq 0} \left( e^{\eta\alpha t} \left(B_\sigma(t) + 1\right)^{-\eta} \right)\right| \, \tau < +\infty \right] \leq \E_1\left( e^{\eta \alpha \tau} | \, \tau < +\infty \right)\\
\times\left( 1 + \E_1 \left. \left[ \sup_{t \geq 0} \left( e^{\eta\alpha t} \left(B_\sigma(t) + 1\right)^{-\eta} \right) \right|\, \tau = +\infty \right] \right).
\end{multline*}
In view of Lemma~\ref{lemma:infinity}, the proof of Proposition~\protect\ref{lemma:births}
will be finished if one can prove that
\[\E_1\left( e^{\eta \alpha \tau} | \, \tau < +\infty \right) < +\infty, \]
which actually comes from the following decomposition: under $\P_1(\, \cdot\, |\, \tau
< +\infty)$, the random variable  $\tau$ can be written as 
\[ \tau = \sum_{k=1}^{1+G} (T_k + E_{\mu_Z,k}) \]
where $G$ is a geometric random variable with parameter $q = \nu / \mu_Z$, $(T_k)$ is an 
i.i.d.\ sequence with the same distribution as the extinction time of a
$(p,\lambda)$-branching process starting with one particle and conditioned on extinction
and $(E_{\mu_Z,k})$ are i.i.d.\ exponential random variables with parameter $\mu_Z$. 

Since $q$ is the probability of extinction of a  $(p,\lambda)$-branching process started
with one particle, $G+1$ represents the number of times $(Z(t))$ hits $0$ before going to
infinity. This representation entails  
\[
\E_1 \left( e^{\eta \alpha \tau} \,|\, \tau < +\infty \right) = \E\left( \gamma(\eta)^{G+1} \right) \ \text{ where }\ \gamma(\eta) = \E\left( e^{\eta \alpha (T_1 + E_{\mu_Z,1})} \right).
\]

A $(p,\lambda)$-branching process conditioned on extinction is actually a
$(1-p,\lambda)$-branching process. See again Athreya and Ney~\cite{Athreya72:0}. Thus
$T_1$ satisfies the following recursive distributional equation: 
\[ T_1 \stackrel{\text{dist.}}{=} E_\lambda + \ind{\xi=2}(T_1 \vee T_2), \]
where $\P(\xi=2) = 1-p$ and $E_\lambda$ is an exponential random variable with
parameter~$\lambda$. This equation yields 
\[ 
\P(T_1 \geq t) \leq e^{-\lambda t} + 2\lambda (1-p) \int_0^t \P(T_1 \geq t-u) e^{-\lambda u}\, du, \]
and Gronwall's Lemma applied to the function $t \mapsto \exp(\lambda t)\P(T_1 \geq t)$
gives that
\[
\P(T_1\geq t)\leq e^{(\lambda -2\lambda p)t}=e^{(\nu-\mu_Z)t}
\]
hence for any $0 < \eta < 1$,
\[ \E_1(e^{\eta \alpha T_1}) \leq \frac{1}{1-\eta}. \]
Since $G$ is a geometric random variable with parameter $q$, $\E \left(
\gamma(\eta)^G\right)$ is finite if and only if $\gamma(\eta) < q$. Since finally 
\[
\gamma(\eta) = \frac{\mu_Z}{\mu_Z-\eta\alpha} \E\left( e^{\eta \alpha T_1} \right) \leq \frac{\mu_Z}{(1-\eta)(\mu_Z - \eta \alpha)},
\]
one can easily check that $\gamma(\eta) < q$ for $\eta < \eta^*(\nu/\mu_Z)$ as defined
by Equation~\eqref{eq:eta}, which concludes the proof of Proposition~\ref{lemma:births}. 

\providecommand{\bysame}{\leavevmode\hbox to3em{\hrulefill}\thinspace}
\providecommand{\MR}{\relax\ifhmode\unskip\space\fi MR }
\providecommand{\MRhref}[2]{%
  \href{http://www.ams.org/mathscinet-getitem?mr=#1}{#2}
}
\providecommand{\href}[2]{#2}


\begin{thebibliography}{10}

\bibitem{Aldous}
David Aldous and Jim Pitman, \emph{Tree-valued {M}arkov chains derived from
  {G}alton-{W}atson processes}, Annales de l'Institut Henri Poincar\'e.
  Probabilit\'es et Statistiques \textbf{34} (1998), no.~5, 637--686.

\bibitem{Alsmeyer93:0}
Gerold Alsmeyer, \emph{{On the Galton-Watson Predator-Prey Process}}, Annals of
  Applied Probability \textbf{3} (1993), no.~1, 198--211.

\bibitem{Athreya72:0}
K.~B. Athreya and P.~E. Ney, \emph{Branching processes}, Springer, 1972.

\bibitem{Bramson}
Maury Bramson, \emph{Stability of queueing networks}, Lecture Notes in
  Mathematics, vol. 1950, Springer, Berlin, 2008, Lectures from the 36th
  Probability Summer School held in Saint-Flour, July 2--15, 2006.

\bibitem{Chen:14}
Hong Chen and David~D. Yao, \emph{Fundamentals of queueing networks},
  Springer-Verlag, New York, 2001, Performance, asymptotics, and optimization,
  Stochastic Modelling and Applied Probability.

\bibitem{Dang}
T.~D. Dang, R.~Pereczes, and S.~Moln\'ar, \emph{Modeling the population of
  file-sharing peer-to-peer networks with branching processes}, IEEE Symposium
  on Computers and Communications (ISCC'07) (Aveiro, Portugal), July 2007.

\bibitem{Kingman}
J.~F.~C. Kingman, \emph{The first birth problem for an age-dependent branching
  process.}, Annals of Probability \textbf{3} (1975), no.~5, 790--801.

\bibitem{Lamperti67:1}
John Lamperti, \emph{Continuous-state branching processes}, Bulletin of the
  American Mathematical Society \textbf{73} (1967), 382--386.

\bibitem{Leskela}
L.~Leskel\"a, \emph{Stochastic relations of random variables and processes}, J.
  Theor. Probab. (2010), To appear.

\bibitem{Massoulie05:0}
Laurent Massouli\'{e} and Milan Vojnovi\'{c}, \emph{Coupon replication
  systems}, Proceedings of SIGMETRICS'05 (Banff, Alberta, Canada), no.~1, June
  2005, pp.~2--13.

\bibitem{Nerman81:0}
Olle Nerman, \emph{On the convergence of supercritical general ({C}-{M}-{J})
  branching processes}, Z. Wahrscheinlichkeitstheorie verw. Gebiete \textbf{57}
  (1981), 365--395.

\bibitem{Neveu}
J.~Neveu, \emph{Erasing a branching tree}, Advances in Applied Probability
  (1986), no.~suppl., 101--108.

\bibitem{Queija08:0}
R.~{N\'u\~nez}-Queija and B.~J. Prabhu, \emph{Scaling laws for file
  dissemination in {P2P} networks with random contacts}, Proceedings of IWQoS,
  2008.

\bibitem{Qiu04:0}
Dongyu Qiu and R.~Srikant, \emph{Modeling and performance analysis of
  bittorrent-like peer-to-peer networks}, SIGCOMM '04: Proceedings of the 2004
  conference on Applications, technologies, architectures, and protocols for
  computer communications (New York, NY, USA), ACM, 2004, pp.~367--378.

\bibitem{Robert03}
Philippe Robert, \emph{Stochastic networks and queues}, Stochastic Modelling
  and Applied Probability Series, vol.~52, Springer, New-York, June 2003.

\bibitem{Rogers2}
L.~C.~G. Rogers and David Williams, \emph{Diffusions, {M}arkov processes, and
  martingales. {V}ol. 2: It\^{o} calculus}, John Wiley \& Sons Inc., New York,
  1987.

\bibitem{Seneta70:0}
E.~Seneta, \emph{On the supercritical branching process with immigration},
  Mathematical Biosciences \textbf{7} (1970), 9--14.

\bibitem{Simatos08:0}
Florian Simatos, Philippe Robert, and Fabrice Guillemin, \emph{A queueing
  system for modeling a file sharing principle}, Proceedings of SIGMETRICS'08
  (New York, NY, USA), ACM, 2008, pp.~181--192.

\bibitem{Susitaival06:0}
Riikka Susitaival, Samuli Aalto, and Jorma Virtamo, \emph{Analyzing the
  dynamics and resource usage of {P2P} file sharing by a spatio-temporal
  model}, International Workshop on {P2P} for High Performance Comptutational
  Sciences, 2006.

\bibitem{Williams91:0}
David Williams, \emph{Probability with martingales}, Cambridge University
  Press, 1991.

\bibitem{Yang04:0}
Xiangying Yang and Gustavo de~Veciana, \emph{Service capacity of peer to peer
  networks}, Proceedings of IEEE Infocom'04, ACM, 2004, pp.~2242--2252.

\end{thebibliography}
\end{document}